\documentclass{amsart}
\usepackage{amsmath}
\usepackage{amsthm}
\usepackage{amssymb}

\usepackage{tikz}
\usetikzlibrary{matrix}

      \theoremstyle{plain}
      \newtheorem{theorem}{Theorem}
      \newtheorem{lemma}[theorem]{Lemma}

      \theoremstyle{definition}

      \theoremstyle{remark}

      \theoremstyle{plain}
      \newtheorem*{theorem*}{Theorem}
      \newtheorem*{lemma*}{Lemma}
      \newtheorem*{corollary*}{Corollary}
      \newtheorem*{proposition*}{Proposition}
      \newtheorem*{conjecture*}{Conjecture}
      \newtheorem*{question*}{Question}
      \newtheorem*{claim*}{Claim}

      \theoremstyle{definition}
      \newtheorem*{definition*}{Definition}
      \newtheorem*{example*}{Example}
      \newtheorem*{game*}{Game}

      \theoremstyle{remark}
      \newtheorem*{remark*}{Remark}

\begin{document}

\title[Idempotent bonding relations are nontrivial iff they satisfy Gamma]{Idempotent bonding relations are nontrivial if and only
if they satisfy condition \(\Gamma\)}



\author{Steven Clontz}
\address{Department of Mathematics and Statistics, University
of North Carolina at Charlotte,
Charlotte, NC 28223}
\email{steven.clontz@gmail.com}
\author{Scott Varagona}
\address{Department of Biology, Chemistry \& Mathematics,
University of Montevallo, Montevallo, AL 35115}
\email{svaragona@montevallo.edu}


\newcommand{\<}{\langle}
\renewcommand{\>}{\rangle}
\newcommand{\rest}{\restriction}

\begin{abstract}
  A relation \(f\subseteq X^2\) satisfies condition \(\Gamma\) if
  there exist distinct \(x,y\in X\) with 
  \(\<x,x\>,\<x,y\>,\<y,y\>\in f\).
  The authors improve a previous result by characterizing
  nontrivial idempotent bonding relations on 
  compact Hausdorff spaces as those satisfying 
  condition \(\Gamma\).
\end{abstract}

\maketitle

\section{Preliminaries}

Assume \(X\) is always a compact Hausdorff space.

Let a relation \(f\subseteq X^2\) be full if
\(\forall x\in X \exists y\in X(\<x,y\>\in f)\).
We define a bonding relation \(f\subseteq X^2\) 
on \(X\) to be a full relation which is a closed subset
of \(X^2\). Such relations 
are often alternately characterized as upper-semicontinuous (u.s.c.)
maps, which are continuous functions from \(X\) to the space 
\(H(X)\) of nonempty closed
subsets of \(X\). As such let \(f(x)=\{y\in X:\<x,y\>\in f\}\),
\(f[A]=\{y\in X:\exists x\in A(\<x,y\>\in f)\}\), and
\(
  f^2
    =
  f\circ f=\{\<x,z\>\in X^2
    :
  \exists y\in X(\<x,y\>,\<y,z\>\in f)\}
\), that is, \(f^2(x)=f[f(x)]\).

A relation is idempotent if \(f=f^2\). 
It is surjective if for each \(y\in X\),
there exists \(x\in X\) where \(\<x,y\>\in f\).
For \(A\subseteq X\), let \(f\rest A=\{\<x,y\>\in f:x\in A\}\)
be the restriction of \(f\) to \(A\). Note that if \(f\) is
idempotent then \(f\rest f(x)\) is surjective (onto \(f(x)\))
for all \(x\in X\).
Let \(\iota=\{\<x,x\>:x\in X\}\) be the identity relation.
We say a bonding relation \(f\) is nontrivial if for some \(x\in X\),
\(f\rest f(x)\not=\iota\rest f(x)\). A single-valued bonding
relation satisfies \(|f(x)|=1\) for all \(x\in X\).

It's important to note that
if \(f\) is an idempotent surjective
single-valued bonding relation,
then \(f=\iota\) and thus is trivial.
Likewise, every trivial idempotent surjection is the
single-valued identity.

However there are trivial idempotent bonding relations besides
the identity: take for instance 
\(t\subseteq\{0,1,2\}^2\) defined by
\(t=\{\<0,0\>,\<1,1\>,\<2,0\>,\<2,1\>\}\). Then
\(t\rest f(2)=t\rest\{0,1\}=\iota\); of
course, \(t\) fails to map to \(2\) and is not
surjective. By connecting the dots the reader may sketch
a version of \(t\) defined for the closed interval 
\([0,2]\subseteq\mathbb R\).

Say that \(f\) satisfies condition \(\Gamma\) if there exist 
distinct \(x,y\in X\) such that \(\<x,x\>,\) 
\(\<x,y\>,\) \(\<y,y\>\in f\).
The authors will show that an idempotent bonding relation
is nontrivial if and only if it satisfies condition \(\Gamma\).
This note answers their question in \cite{CV} by
generalizing their result on interval-valued idempotent relations
defined on the closed interval \([0,1]\subseteq\mathbb R\).

\section{Main Result}



  

\begin{lemma}
  Every nontrivial idempotent bonding relation \(f\) 
  contains two points
  \(\<x,x\>\) and \(\<y,x\>\)
  for distinct \(x,y\in X\).
\end{lemma}

\begin{proof}
  Note first that if \(\iota\subsetneq f\), then the lemma follows
  immediately. So
  let \(x_0\in X\) be a point where \(\<x_0,x_0\>\not\in f\).
  
  Suppose \(x_i\) is defined for \(i\leq n\) such that
  \(\<x_i,x_j\>\in f\) if and only if \(i<j\).
  So we may choose \(x_{n+1}\) distinct from \(x_i\) for \(i\leq n\) 
  such that \(\<x_n,x_{n+1}\>\in f\). If \(\<x_{n+1},x_{n+1}\>\in f\),
  then the lemma is satisfied by \(x=x_{n+1}\) and \(y=x_n\).
  Note that by idempotence, \(\<x_{n+1},x_i\>\not\in f\)
  for \(i\leq n\) as
  otherwise \(x_i\in f(x_{n+1})\subseteq f(f(x_n))=f(x_n)\)
  contradicting \(\<x_n,x_i\>\not\in f\).
  
  Since \(\{x_n:n<\omega\}\) is an infinite set in a compact
  Hausdorff space, it has a limit point \(x_\omega\). Note then that
  for any open neighborhood \(U\) of \(x_\omega\), \(U\) contains
  infinitely many \(x_n\), so choose \(i<j\) such that
  \(x_i,x_j\in U\). Then, it follows that the
  basic open neighborhood \(U^2\) of \(\<x_\omega,x_\omega\>\) 
  contains \(\<x_i,x_j\>\). Thus \(\<x_\omega,x_\omega\>\) is a limit
  point of \(\{\<x_i,x_j\>:i<j<\omega\}\subseteq f\), and as \(f\)
  is closed, \(\<x_\omega,x_\omega\>\) belongs to \(f\). Then since 
  \(x_\omega\not=x_0\) (as \(\<x_0,x_0\>\not\in f\)), 
  we may similarly show
  \(\<x_0,x_\omega\>\) is a limit point of 
  \(\{\<x_0,x_n\>:0<n<\omega\}\subseteq f\), and therefore
  \(\<x_0,x_\omega\>\in f\). The lemma is now witnessed by
  \(x=x_\omega\) and \(y=x_0\).
\end{proof}

\begin{lemma}
  Suppose \(\<x,x\>,\<x,y\>\in f\) for distinct \(x,y\in X\)
  and an idempotent bonding relation \(f\). Then \(f\) satisfies
  condition \(\Gamma\).
\end{lemma}

\begin{proof}
  Let \(z_0=y\). If \(\<y,y\>=\<z_0,z_0\>\in f\), we are done.
  
  Suppose \(z_i\) is defined for \(i\leq n\) such that
  \(\<z_i,z_j\>\in f\) if and only if \(i<j\), and 
  \(\<x,z_i\>\in f\) for \(i\leq n\).
  So we may choose \(z_{n+1}\) distinct from \(z_i\) for \(i\leq n\) 
  such that \(\<z_n,z_{n+1}\>\in f\). Note that \(\<x,z_{n+1}\>\in f\)
  since \(\<x,z_n\>,\<z_n,z_{n+1}\>\in f\) and thus
  \(z_{n+1}\in f(z_n)\subseteq f(f(x))=f(x)\).
  If \(\<z_{n+1},z_{n+1}\>\in f\),
  then the condition \(\Gamma\) is witnessed by 
  \(\<x,x\>,\<x,z_{n+1}\>,\<z_{n+1},z_{n+1}\>\).
  On the other hand, \(\<z_{n+1},z_i\>\not\in f\)
  for \(i\leq n\) as otherwise by idempotence
  \(z_i\in f(z_{n+1})\subseteq f(f(z_n))=f(z_n)\)
  contradicting \(\<z_n,z_i\>\not\in f\).
  Similarly, \(\<z_n,x\>\not\in f\) as otherwise
  \(\<z_n,x\>,\<x,z_n\>\in f\Rightarrow\<z_n,z_n\>\in f\).
  
  Since \(\{z_n:n<\omega\}\) is an infinite set in a compact
  Hausdorff space, it has a limit point \(z\). Note then that
  for any open neighborhood \(U\) of \(z\), \(U\) contains
  infinitely many \(z_n\), so choose \(i<j\) such that
  \(z_i,z_j\in U\). Then, it follows that the
  basic open neighborhood \(U^2\) of \(\<z,z\>\) contains
  \(\<z_i,z_j\>\). Thus \(\<z,z\>\) is a limit
  point of \(\{\<z_i,z_j\>:i<j<\omega\}\subseteq f\), and as \(f\)
  is closed, \(\<z,z\>\) belongs to \(f\). We may similarly show
  \(\<x,z\>\) is a limit point of 
  \(\{\<x,z_n\>:0<n<\omega\}\subseteq f\), and therefore
  \(\<x,z\>\in f\).
  
  We know \(x\not=z\) since otherwise 
  \(\{\<z_0,z_{n+1}\>:n<\omega\}\subseteq f\)
  would imply its limit
  \(\<z_0,z\>=\<z_0,x\>\in f\), which was disproved above.
  Therefore \(\<x,x\>,\<x,z\>,\<z,z\>\in f\) witness condition
  \(\Gamma\).
\end{proof}

\begin{lemma}
  The inverse of an idempotent relation
  is also an idempotent relation.
\end{lemma}

\begin{proof}
  \((f^{-1})^2=(f^2)^{-1}=f^{-1}\).
\end{proof}

\begin{theorem}
  An idempotent bonding relation is nontrivial if and only
  if it satisfies condition \(\Gamma\) if and only if
  it contains two points \(\<x,x\>,\<x,y\>\).
\end{theorem}

\begin{proof}
  Obviously, if a bonding relation \(f\) has condition
  \(\Gamma\) then it contains two points \(\<x,x\>,\<x,y\>\).
  It then follows from those two points that
  \(f\rest f(x)\) is not the identity, and therefore
  \(f\) is nontrivial. 
  
  If \(f\) is nontrivial idempotent, then apply Lemma 1
  to obtain the points \(\<x,x\>,\<y,x\>\in f\). Then
  \(\<x,x\>,\<x,y\>\in f^{-1}\), which is idempotent
  by Lemma 3. So Lemma 2 may be applied to show
  that \(f^{-1}\) has condition \(\Gamma\), and therefore
  so does \(f\).
\end{proof}

\section{An Application}

The authors used \(f\)'s condition \(\Gamma\) in \cite{CV} to
show that for an ordinal \(\alpha\),
the inverse limit \(\varprojlim\{I,f,\alpha\}\)
is metrizable if and only if \(\alpha\) is countable. 
Using a few unpublished results of the first author
along with the main result of this note, this in fact 
generalizes to the following theorem:

\begin{theorem}
  Let \(f\) be a nontrivial bonding relation on a compact
  metrizable space \(X\), and let \(L\) be an arbitrary
  total order. Then the inverse limit
  \(\varprojlim\{X,f,L\}\) is metrizable if and only if
  it is Corson compact if and only if
  \(L\) is countable.
\end{theorem}

\begin{proof}
  If \(L\) is countable then the subspace 
  \(\varprojlim\{X,f,L\}\) of the metrizable space
  \(X^L\) is of course metrizable. If \(L\) is uncountable,
  note that \(\varprojlim\{X,f,L\}\) contains the
  subspace \(\varprojlim\{2,\gamma,L\}\) where
  \(\gamma=\{\<0,0\>,\<0,1\>,\<1,1\>\}\). It may
  be shown that \(\varprojlim\{2,\gamma,L\}\) is homeomorphic
  to a compact linearly ordered topological space
  \(\check L\) which is metrizable if and only if
  it is Corson compact if and only if it is second-countable.
  The result follows by showing that \(\check L\)
  is second-countable if and only if \(L\) is countable.
\end{proof}

\section{Acknowledgments}

The authors would like to thank Jonathan Meddaugh for his thoughtful
input on this topic during the 2016 Spring Topology and Dynamics 
Conference, and more generally thank the organizers of all STDC
conferences over the past fifty years for giving young researchers
an environment to share ideas and collaborate on problems.


\begin{thebibliography}{9}
\bibitem{CV}
Steven Clontz and Scott Varagona, Destruction of Metrizability in Generalized Inverse Limits, Top. Proc. 48 (2016) pp. 289-297.
\end{thebibliography}
\end{document}